\documentclass[a4paper,12pt]{amsart}
\usepackage{amsfonts}
\usepackage{mathrsfs}
\usepackage{amsmath,amssymb,latexsym,amsfonts,amscd}
\usepackage{amsmath}

\title[Finiteness of Calabi-Yau quasismooth W.C.I.]{Finiteness of Calabi-Yau quasismooth weighted complete intersections}
\author{Jheng-Jie Chen}

\address{\rm Department of Mathematics, National Taiwan University, Taipei, 106,
Taiwan}
\email{d94221006@gmail.com}

\newcommand{\bQ}{{\mathbb Q}}
\newcommand{\bP}{{\mathbb P}}

\newcommand\PPP{{\mathbb{P}}}

\newtheorem{thm}{Theorem}[section]
\newtheorem{lem}[thm]{Lemma}

\newtheorem{prop}[thm]{Proposition}
\newtheorem{claim}[thm]{Claim}

\theoremstyle{definition}
\newtheorem{defn}[thm]{Definition}
\newtheorem{setup}[thm]{}

\newtheorem{exmp}[thm]{Example}
\newtheorem{conj}[thm]{Conjecture}
\newtheorem{rem}[thm]{Remark}
\theoremstyle{remark}

\begin{document}
\maketitle
\scriptsize
\noindent $\mathrm{ABSTRACT.}$
We prove that there exist only finitely many families of Calabi-Yau quasismooth weighted complete intersections with every fixed dimension $m$.
This generalizes a result of Johnson and Koll\'{a}r to higher codimensions.

\normalsize
\section{\bf Introduction}

Examples of complete intersections in weighted projective spaces are interesting and useful when studying higher dimensional birational geometry
(cf. \cite{elephant,SGM,Icecream,explicit1,AIM,HM,Takuzo}).
In \cite{C3-f,YPG,Fletcher}, Reid and Fletcher give several famous lists of families of three dimensional well-formed quasismooth weighted complete intersections with terminal singularities.
In \cite{CCC}, by basket analysis and Reid's table method, it is proved that these lists are complete.

In this article, we are interested in the finiteness of families of well-formed quasismooth weighted complete intersections in general.
In the case of Calabi-Yau quasismooth hypersurfaces,
Johnson and Koll\'{a}r prove the finiteness of such families for every fixed dimension in \cite{KJ}.
The aim of this note is to generalize it to higher codimension cases.
Combining this with boundedness for codimension (Theorem \ref{cod}), we derive Theorem \ref{CY}.

\begin{thm}\label{CY}
For any positive integer $m$, there are only finitely many families of Calabi-Yau quasismooth weighted complete intersections of dimension $m$.
\end{thm}

In \cite{KJ2,KJ}, Johnson and Koll\'{a}r give complete classifications of anticanonical embedded Fano quasismooth
hypersurfaces in weighted projective spaces in dimension two and three.
There are infinitely many such families for these cases (see Example \ref{JKex}).
Because of this, it is natural to ask the Borisov-Alexeev-Borisov conjecture in quasismooth weighted complete intersections case.
\begin{conj}\label{BAB}
For fixed positive number $m,\epsilon$, and a negative integer $\alpha$,
there exist only finitely many of families of quasismooth weighted complete intersections having only $\epsilon-$klt singularities and dimension $m$, amplitude $\alpha$.
\end{conj}

Let $X_{d_1,...,d_c}\subset \PPP(a_0,...,a_n)$ be a weighted complete intersection with amplitude $\alpha:=\sum_{j=1}^c d_j-\sum_{i=0}^n a_i\ge -1$.
We say that it is normalized if $a_0,...,a_n,d_1,...,d_c$ are positive integers with
\[a_0\leq a_1\leq \cdots\leq a_n\mbox{ and }\ d_1\leq \cdots\leq d_c. \]
From studying quasismooth behaviors at two points
$\mbox{P}_n$ and $\mbox{P}_{\dim X+1}$ in $\PPP(a_0,...,a_n)$, we establish an effective upper bound $(\dim X+1)\delta$ for $a_n$ (see Proposition \ref{comparison}),
where  $\delta$ is the integer $a_0+\cdots+a_{\dim X}+\alpha$.
This enables us to apply the sandwich argument as Johnson and Koll\'{a}r did in hypersurface case and thus prove Theorem \ref{CY}.

Note that we have $a_n<d_c$ by the quasismoothness at $\mbox{P}_n$ (see \cite[Lemma 18.14]{Fletcher} or Proposition \ref{qs}).
Another application of Proposition \ref{comparison} is that it provides an upper bound for $d_c$ in terms of dimension $m$,
amplitude $\alpha$ and a lower bound for volume $K_X^{m}$ (resp. anti-volume $-K_X^{m}$) if $\alpha>0$ (resp. if $\alpha<0$).

\begin{thm}\label{bound}
For given integers $m\geq 2$,$\ \alpha,c$, and positive rational numbers $b,\epsilon$.
Let $X=X_{d_1,...,d_c}\subset\PPP(a_0,...,a_n)$ be a family of $m$-dimensional normalized quasismooth weighted complete intersections with amplitude $\alpha> 0$ and volume $K_X^m>b$
(resp. anti-volume $-K_X^m>b$ and put $\epsilon=1$ provided $\alpha=-1$).
If $X$ is Fano with $\alpha\leq -2$, we require that anti-volume $-K_X^m>b$ and $X$ has only $\epsilon-$klt singularities.
Then $d_c$ is bounded from above by
\[  \left\{
\begin{array}{ll}
\frac{m+2}{b}\cdot ((m+1)\alpha^{m}(\frac{c+\alpha+m+1}{c})^c +b\alpha) & \mbox{if } \alpha > 0, \vspace{0.2cm}\\
\frac{m+2\epsilon}{b\epsilon}\cdot ((m+1)(-\alpha)^{m} (\frac{c+m+1}{c})^c +b\alpha) & \mbox{if }\alpha <0.
\end{array} \right. \]
\end{thm}

In particular, this provides the finite possible families of three dimensional weighted complete intersections with terminal singularities
since we have the lower bound for volume (resp. anti-volume if $X$ is Fano)(see Remark \ref{lists}).

As an application of Theorem \ref{bound}, proving Conjecture \ref{BAB} is equivalent to giving a universal bound of one of the following
\begin{enumerate}
\item there exists an integer $r=r(m,\alpha,\epsilon)$ only depending on $m,\epsilon$ such that $r\cdot K_X$ is Cartier;
\item there exists a positive number $b=b(m,\alpha,\epsilon)$ only depending on $m,\epsilon$ such that $-K_X^m\geq b$;
\item there exists an integer $\beta=\beta(m,\alpha,\epsilon)$ only depending on $m,\epsilon$ such that $a_m\leq \beta$,
\end{enumerate}
for all quasismooth $X=X_{d_1,...,d_c}\subset \PPP(a_0,...,a_n)$ with only $\epsilon-$klt singularities and dimension $m$, amplitude $\alpha$.
Conjecture \ref{BAB} remains open even in dimension three.

\begin{setup}{\bf Acknowledgement}.
The author was partially supported by NCTS/TPE and the National Science Council of Taiwan.
The author expresses his gratitude to Professor Jungkai Alfred Chen for suggesting the question and extensively helpful and invaluable discussions.
He is also grateful to Professors Miles Reid and J\'{a}nos Koll\'{a}r and Chin-Lung Wang and Hui-Wen Lin for useful discussions and information and warm encouragements.
He thanks Professors Gavin Brown and Takuzo Okada for answering many questions through e-mails.
\end{setup}

\section{\bf Preliminaries and notations}
We fix the notations here.

\begin{defn}
Given positive integers $a_0,...,a_n$, $\PPP (a_0,...,a_n)$ is the weighted projective space Proj$(S)$
where $S=\mathbb{C}[x_0,...,x_n]$ is the graded ring with deg $x_i=a_i$ for all $i$.
\end{defn}
As in \cite{Fletcher}, we may assume $\PPP (a_0,...,a_n)$ is well-formed, i.e, great common divisor of $a_0,...,\hat{a_i},...,a_n$ is $1$ for all $i=0,...,n$.
Also, $\PPP (a_0,...,a_n)$ is the quotient $\mathbb{C}^{n+1}-\{(0,...,0)\}/\mathbb{C}^*$ under the equivalent relations $(x_0,...,x_n)\sim (\lambda^{a_0}x_0,...,\lambda^{a_n} x_n)$ for all $\lambda\in \mathbb{C}^*$.
Denote $\pi:\mathbb{C}^{n+1}-\{(0,...,0)\}\rightarrow \PPP (a_0,...,a_n)$ to be the quotient map.

\begin{defn}
Let $c$ be a positive integer and $d_1,...,d_c$ be positive integers.
Let $f_1(x_0,...,x_n),...,f_c(x_0,...,x_n)$ be general homogeneous polynomials of degree $d_1,...,d_c$ under the weights deg $x_i=a_i$.
A family of weighted complete intersections $X:=X_{d_1,...,d_c}$ is defined to be a subvariety $(f_1=\cdots =f_c=0)$ in $\PPP (a_0,...,a_n)$.
$X$ is quasismooth if the affine cone $\pi^{-1}(X)\cup\{0\}$ is smooth away from zero.
\end{defn}

Suppose $X$ is an intersection of a linear cone with a subvariety in $\PPP$, i.e, $d_j=a_i$ for some $i,j$, so $f_j=x_i+\mbox{others}$.
Then $X\subset\PPP$ is isomorphic to $X_{d_1,...,\hat{d_j},...,d_c}\subset \PPP(a_0,...,\hat{a_i},...,a_n)$.
In this note, we always assume $X$ is not an intersection of a linear cone with another subvariety. Also, we use the conventions:
\[n:=\dim \PPP,\ m:=\dim X,\ c:=\mbox{codim}(X,\PPP),\mbox{ thus }\ m+c=n.\]

By renumbering the indices, we assume it is normalized, i.e,
\[a_0\leq a_1\leq\cdots \leq a_n \mbox{ and } d_1\leq \cdots \leq d_c.\]
For a normalized weighted complete intersection, for each $j=1,...,c$, put
\[\delta_j:= d_j -a_{j+m},\ \delta:=\sum_{j=1}^c\delta_j, \mbox{ and } \alpha:=\sum_{j=1}^c d_j-\sum_{i=0}^n a_i,\]
where $\alpha$ is called the $\textit{amplitude}$ of $X$.
For every non-empty subset $E$ of $\{0,1,...,n\}$, we define
the $|E|-1$ dimensional stratum $\mbox{P}_E:=\{(x_0,...,x_n)\in \PPP (a_0,...,a_n)|x_i =0 \mbox{\ for all\ }i\notin E\}$.
We say that $X\subset \PPP$ is well-formed if $\PPP$ is well-formed and $X$ contains no codimension $c+1$ singular strata of $\PPP$.
From \cite[Theorem 6.16]{Fletcher}, $X$ is well-formed if $X\subset \PPP$ is quasismooth with dimension greater than $2$.
If $X$ is well-formed and quasismooth, the dualizing sheaf $\omega_X=\mathcal{O}_X(K_X)\simeq \mathcal{O}_X(\alpha)$ (see \cite[Theorem 3.3.4]{WPS}).

The following necessary condition for quasismoothness is helpful for our discussions.

\begin{prop} [Fletcher] \label{quasismooth}
Let $X=X_{d_1,...,d_c}\subset \PPP(a_0,...,a_n)$ be a
quasismooth weighted complete intersection. For every subset $E\subset \{0,1,...,n\}$, we define $\rho_E:=\min\{c,|E|\}$. Then one of following holds:
\begin{itemize}
\item[{\bf  (1)}] there exist distinct integers $p_1,...,p_{\rho_E}$ which are elements of $1,...,c$ such that for all $j$, $f_{p_j}=\Pi_{i\in E}\ x_i^{k_{j,i}}+\mbox{others}$ ;
\item[{\bf  (2)}] there exists a permutation $p_1,...,p_c$ of $1,...,c$, and there exist distinct integers $e_{l+1},...,e_c\in \{0,...,n\}-E$ for some integer $l\geq 0$ satisfying
\[  \left\{
\begin{array}{ll}
f_{p_j}=\Pi_{i\in E}\ x_i^{k_{j,i}}+\mbox{others} &\mbox{ for } j=1,...,l, \\
f_{p_j}= x_{e_j}\Pi_{i\in E}\ x_i^{k_{j,i}}+\mbox{others}  &\mbox{ for } j=l+1,...,c.\ \\
\end{array} \right. \]
\end{itemize}
\end{prop}
\begin{proof}
We briefly explain the proof.
For every non-empty subset $E$ of $\{0,1,...,n\}$, we consider the intersection $\mbox{P}_{E}\cap X$.
If the set is empty, by counting dimension, condition (1) holds.
Suppose the set is non-empty.
Since quasismoothness shows that the Jacobian matrix on general points in the affine cone of $\mbox{P}_{E}\cap X$ is of full rank, condition (2) holds.
For details, see \cite{Fletcher,CCC}.
\end{proof}

By counting degrees and studying numerical conditions to some strata, we obtain the following.
\begin{prop}[\cite{Fletcher}] \label{qs} Let $X_{d_1,\cdots,d_c} \subset \bP(a_0,\cdots,a_{n})$ be a quasismooth complete intersection
and is not an intersection of a linear cone with another subvariety
(i.e. $d_j\neq a_i$ for all $i,j$). Then we have
\begin{enumerate}
\item If $a_{t} > d_1$ for some $t \ge 0$, then $a_t|d_j$ for some $j$.
In particular, $\delta_c  \geq a_n$ in this situation.

\item For $t=1,2...,c$, we get $\delta_{t}>0$.
\end{enumerate}
\end{prop}

\begin{thm} [\cite{CCC}]\label{cod}
If $X=X_{d_1,...,d_c}\subset \PPP(a_0,...,a_n)$ is a family of quasismooth weighted complete intersections with amplitude $\alpha$, dimension $m$ and codimension $c$,
which is not an intersection of a linear cone with another subvariety.
Here amplitude $\alpha$ is defined to be the integer $\sum_{j=1}^c d_j-\sum_{i=0}^n a_i$.
Then the codimension $c$ has the upper bound $m+\alpha+1$ (resp. $m$) if amplitude $\alpha\ge 0$ (resp. $\alpha <0$).
\end{thm}

\begin{defn}
Given a positive number $\epsilon\leq 1$. A normal projective variety $X$ has only $\epsilon-klt$ singularities if it satisfies the following conditions:
\begin{enumerate}
\item[(1)] The Weil divisor $rK_X$ is Cartier for some positive integer $r$.
\item[(2)] For a resolution $f:Y\to X$ with exceptional divisors $E_1,...,E_s$, then we have $K_Y= f^*(K_X)+\sum_{i=1}^s q_i E_i$ with $q_i>\epsilon-1$ for all $i$.
\end{enumerate}
\end{defn}

\section{\bf A bound for $a_n$ and its application}
In this section,
for each normalized quasismooth weighted complete intersection $X=X_{d_1,...,d_c}\subset \PPP (a_0,...,a_n)$ with amplitude $\alpha\geq -1$,
we give an upper bound $(m+1) \delta$ for $a_n$, where $m:=\dim X$.
We shall see that this provides the upper bound for $d_c$ in terms of a lower bound of volume (resp. anti-canonical volume in Fano case).

By studying quasismooth behavior of the strata $\mbox{P}_{n}$ and $\mbox{P}_{m+1}$, there is an upper bound of $a_n$ in terms of $\delta$.
\begin{prop} \label{comparison}
Let $X=X_{d_1,...,d_c}\subset \PPP(a_0,....,a_n)$ be a family of quasismooth weighted complete intersections with amplitude $\alpha$ and dimension $m$.
If $\alpha\geq -1$, then the inequality $(m+1)\cdot\delta>a_n$ holds.
Given a positive rational number $\epsilon$.
If X is Fano with only $\epsilon-$klt singularities and $a_n>\frac{m+\epsilon}{m}\cdot\frac{-\alpha}{\epsilon}$, then $\frac{m+\epsilon}{\epsilon}\delta>a_n$.
\end{prop}

\begin{proof}
Suppose that $X_d$ is a hypersurface in weighted projective space with $\alpha \ge -1$.
From the equality $d=\sum_{i=0}^n a_i+\alpha$, there is no $i<n$ with $d=a_n+a_i$.
We are always in the case $d\geq 2a_n$ by Proposition \ref{quasismooth}.

Suppose on the contrary that $X_d$ is Fano hypersurface with $d<2a_n$ and $a_n>\frac{-\alpha}{\epsilon}$.
Then $\mbox{P}_n:=(0,...,0,1)$ belongs to $X$ and condition (2) of Proposition \ref{quasismooth} implies $d=a_n+a_i$ for some $i\neq n$.
By Inverse Function Theorem, $\mbox{P}_n\in X$ is a cyclic quotient point of type
$$\frac{1}{a_n}(a_0,...,\hat{a_i},...,a_{m}).$$
By taking the weighted blow up $\phi:Y\rightarrow X$ at the center $\mbox{P}_n\in X$ with weight $(\frac{a_0}{a_n},...,\hat{\frac{a_i}{a_n}},...,\frac{a_{m}}{a_n})$,
we have $K_Y=\phi^* K_X+qE$, where $E$ is the exceptional divisor and
$q=\sum_{k\neq i,n}\frac{a_k}{a_n}-1=\frac{-\alpha}{a_n}-1<\epsilon -1.$
This contradicts to $\epsilon-$klt assumption.

So we may assume that $c\geq 2$ and $d_c<2a_n$.
We divide it into two parts by comparing $a_{m+1}$ to the number $\frac{m}{m+\epsilon}\cdot a_n$.
Here we put the positive number $\epsilon=1$ if amplitude $\alpha\geq -1$.
From Proposition \ref{qs} above, we get $d_1>a_n$.

Suppose that $a_{m+1}\leq \frac{m}{m+\epsilon}\cdot a_n$.
In this case, we obtain
$$\delta> \delta_1=d_1-a_{m+1}>a_n-\frac{m}{m+\epsilon}\cdot a_n=\frac{\epsilon}{m+\epsilon}\cdot a_n.$$

Suppose that $a_{m+1}> \frac{m}{m+\epsilon}\cdot a_n$ and the condition $a_n\geq \frac{m+\epsilon}{\epsilon}\delta$ holds.
If $m\geq 2$, we obtain
\[
2a_{m+1}>\frac{m-\epsilon}{m+\epsilon}a_n+a_n\geq  \delta+a_n> \delta_c+a_n=d_c.
\]
This implies that the point $\mbox{P}_{m+1}:=(0,...,0,1,0,..,0)$ belongs to $X$.
Quasismoothness at $\mbox{P}_{m+1}\in X$ implies condition (2) of Proposition \ref{quasismooth},
so $\mbox{P}_{m+1}\in X$ is a cyclic quotient point of type $$\frac{1}{a_{m+1}}(b_1,...,b_k,\overline{a_{m+2}},...,\overline{a_n}),$$
where $\{b_1,...,b_k\}$ is a proper subset of $\{a_0,...,a_{m}\}$ and here $\overline{a_i}$ denotes the smallest positive residue of $a_i$ mod $a_{m +1}$.
Here the subset $\{b_1,...,b_k\}$ is empty if and only if $c\geq m+1$.
From the equalities $\sum_{j=1}^c d_j=\sum_{i=0}^n a_i+\alpha$ and $d_j=a_{m+1}+a_{e_j}$ with distinct $e_j\leq m$ for all $j$,
one observes that
\[
0=\sum_{i=1}^k b_i+\sum_{i=m+2}^n \overline{a_{i}}+\alpha.
\]
This is impossible for $\alpha> 0$ case.
It cannot occur when $c<m+1$ in Calabi-Yau case (resp. $c<m$ in Fano case with amplitude $\alpha=-1$).
If $c=m+1$ in Calabi-Yau case (resp. $c=m$ in Fano $\alpha=-1$ case),
then $a_{m+1}=\cdots=a_n$, so $X$ contains the $c$ (resp. $c+1$)-codimensional singular stratum $\mbox{P}_{\{m+1,...,n\}}$ of weighted projective space $\PPP(a_0,...,a_n)$.
This is not well-formed.

For Fano case and $a_n>\frac{m+\epsilon}{m}\cdot\frac{-\alpha}{\epsilon}$,
by taking the weighted blow up $\phi:Y\rightarrow X$ at the center $\mbox{P}_{m+1}\in X$ with weight $$(\frac{b_1}{a_{m+1}},...,\frac{b_k}{a_{m+1}},\frac{\overline{a_{m+2}}}{a_{m+1}},...,\frac{\overline{a_n}}{a_{m+1}}),$$
this shows that $K_Y=\phi^* K_X+qE$, where $E$ is the exceptional divisor and
\begin{align*}
q&=\frac{\sum_{i=1}^k b_i+\sum_{i=m+2}^n \overline{a_{i}}}{a_{m+1}}-1=\frac{-\alpha}{a_{m+1}}-1\\
 &<\epsilon\cdot \frac{m}{m+\epsilon}\cdot\frac{a_n}{a_{m+1}}-1<\epsilon-1.
\end{align*}
So $X$ contains a non-$\epsilon-$klt point $\mbox{P}_{m+1}$. This is the contradiction.
\end{proof}

\begin{exmp}
Given an $\epsilon>0$, $X_d\subset \PPP (1,1,1,1,d-1)$ is a quasismooth hypersurface with $\delta =1$,
amplitude $\alpha=-3$ and a cyclic quotient singularity $\mbox{P}_4$ of type $\frac{1}{d-1}(1,1,1)$ which is not $\epsilon-$klt when $d\ge \frac{3}{\epsilon+1}+1$.
The assumption of being $\epsilon-$klt is necessary.
\end{exmp}

From \cite{CCC}, we observe the following inequalities
\begin{prop}
Let $X_{d_1,...,d_c}\subset \PPP(a_0,...,a_n)$ be normalized quasismooth with amplitude $\alpha$ and dimension $m$. Then
\[\left\{
\begin{array}{ll}
\frac{(\frac{c+\alpha+m+1}{c})^c}{\prod_{i=0}^{m}a_i } \ge
 \frac{\prod_{j=1}^c d_j}{\prod_{i=0}^{n} a_i}=\mathcal{O}_X(1)^{m}=\frac{K_X^{m}}{ \alpha^{m}} & \mbox{if } \alpha > 0, \vspace{0.3cm}\\
\frac{(\frac{c+m+1}{c})^c}{\prod_{i=0}^{m}a_i } \ge
 \frac{\prod_{j=1}^c d_j}{\prod_{i=0}^{n} a_i}=\mathcal{O}_X(1)^{m}=\frac{(-K_X)^{m}}{(-\alpha)^{m}} & \mbox{if } \alpha <0.
\end{array} \right. \]
\end{prop}
This provides an upper bound for $\delta$ in terms of volume $K_X^{m}$ (resp. anti-canonical volume $-K_X^{m}$) if X is of general type (resp. $X$ is Fano).
Indeed,
$$(m+1)N+\alpha\geq a_0+a_1+\cdots+a_{m}+\alpha=\delta,$$
where
\[N:=\left\{
\begin{array}{ll}
 \alpha^{m}(\frac{c+\alpha+m+1}{c})^c/K_X^{m} & \mbox{if } \alpha > 0, \vspace{0.3cm}\\
(-\alpha)^{m}(\frac{c+m+1}{c})^c/(-K_X)^{m} & \mbox{if } \alpha <0.
\end{array} \right. \]
Combining this with Proposition \ref{comparison}, one observes (if $\alpha >0$ or $\alpha=-1$, we put $\epsilon=1$)
$$ \frac{m+2\epsilon}{\epsilon}\cdot ((m+1)N+\alpha)\geq \frac{m+2\epsilon}{\epsilon}\delta>\delta+a_n\geq d_c.$$
In conclusion, for fixed dimension and amplitude $\alpha$ we obtain a bound for $d_c$ in terms of a lower bound for volume as in Theorem \ref{bound}.

\begin{rem}\label{lists}
In \cite{CCC}, we use singular Riemann-Roch formula, so called basket technique and the optimal lower bound for volume $K_X^3\geq \frac{1}{420}$ in general type case (resp. anti-volume $-K_X^3\geq \frac{1}{330}$ and Miyaoka-Yau inequality $-K_X . c_2(X)>0$ in Fano case) to prove several lists for terminal threefold weighted complete intersections provided by Fletcher \cite{Fletcher} are complete.
In these cases, Theorem \ref{bound} gives all the finite possible choices by applying the lower bound for volume or anti-volume directly.
However, the basket technique is more effective since it produced only a few extra examples (compared to Fletcher's lists) needed to be ruled out.
\end{rem}
For given positive integers $m,\alpha$, the following Theorem, proved independently by Hacon and Mckernan, Takayama, Tsuji, provides a universal lower bound for volume $K_X^m$ when $X$ is of general type with given dimension $m$.

\begin{thm}[\cite{HM,Takayama,Tsuji}] \label{pluricano}
For every positive integer $m$, there exists a positive integer $r(m)$ depending only on $m$ such that the pluricanonical map $\phi_{|lK_X|}:X \dashrightarrow \PPP(H^0(X,\mathcal{O}_X(lK_X)))$ is birational for all $m$-dimensional smooth projective variety $X$ of general type and for every integer $l\geq r(m)$.
\end{thm}

Together with Theorem \ref{bound} and Theorem \ref{cod}, this gives the finiteness for general type cases:

\begin{thm}\label{general}
For fixed positive integers $m,\alpha$, there exist only finitely many of families of quasismooth weighted complete intersections with dimension $m$ and amplitude $\alpha$.
\end{thm}

In \cite{Borisov}, Borisov proved Borisov-Alexeev-Borisov conjecture in threefolds with any fixed Gorenstein index case.
In weighted complete intersection cases, this is obvious since the anti-volume $(-K_X)^{m}\geq 1/r^{m}$ where $r$ is a fixed Gorenstein index, i.e, $r\cdot K_X$ is Cartier.
We obtain the finiteness of families of quasismooth weighted complete intersections with fixed Gorenstein index by Theorem \ref{bound}.

\begin{exmp}[Johnson and Koll\'{a}r]\label{JKex}

\noindent In the case of anticanonical embedded quasismooth Fano hypersurfaces,
there exist exactly 48 types of infinite families of the form
\[X_{2k(b_1+b_2+b_3)}\subset \PPP(2,kb_1,kb_2,kb_3,k(b_1+b_2+b_3)-1),\]
for every odd integer $k$.
Here the triple $(b_1,b_2,b_3)$ satisfies $b_4=b_1+b_2+b_3$ for some $K_3$ quasismooth hypersurface
in weighted projective space $\PPP(b_1,b_2,b_3,b_4)$ observed in \cite[II.3.3]{Fletcher} (cf. \cite[Theorem 2.2]{KJ}).
\end{exmp}

Each of 48 infinite families contains a non-$\epsilon$-klt singularity $\mbox{P}_4$ of type
\[\frac{1}{k(b_1+b_2+b_3)-1}(kb_1,kb_2,kb_3)\] for all odd integers $k\ge\frac{(b_1+b_2+b_3)+\epsilon+1}{(b_1+b_2+b_3)(\epsilon+1)}$
by taking a weighted blow up at $\mbox{P}_4\in X$ with weight
\[
(\frac{b_1}{k(b_1+b_2+b_3)-1},\frac{b_2}{k(b_1+b_2+b_3)-1},\frac{b_3}{k(b_1+b_2+b_3)-1}).
\]
This leads Conjecture \ref{BAB}. For more examples, readers can search graded ring database \cite{datebase} provided by Gavin Brown.

\section{\bf Finiteness of Calabi-Yau weighted complete intersections}

In \cite{KJ}, Johnson and Koll\'{a}r proved the finiteness of families of Calabi-Yau quasismooth hypersurfaces in weighted projective spaces with fixed dimension.
In this section, we generalize this to higher codimension cases. 
Our argument basically follows from Johnson and Koll\'{a}r with some modifications.
\begin{rem}
If $X=X_{d_1,...,d_c}\subset \PPP (a_0,...,a_n)$ is quasismooth of dimension $m$, amplitude $\alpha$,
then $X'=X_{d_1,...,d_c}\subset \PPP (1,a_0,...,a_n)$ is also quasismooth of dimension $m+1$, amplitude $\alpha-1$.
Note that the converse does not hold in general.
As a corollary, Theorem \ref{CY} recovers Theorem \ref{general} (see also \cite[Corollary 4.3]{KJ}).
\end{rem}

\begin{proof}[Proof of Theorem \ref{CY}.]
\noindent

Suppose Theorem \ref{CY} is not true for some dimension $m$.
Since the codimension is bounded by $m+1$ by Theorem \ref{cod},
there exist infinite families of Calabi-Yau quasismooth weighted complete intersections of fixed dimension $m$ and fixed codimension $c$,
say $X(t)=X_{d_1(t),...,d_c(t)}\subset \PPP(a_0(t),...,a_n(t))$.
As the setting in \cite{KJ},
by the homogenity we may assume $\sum_{i=0}^n a_i(t)=1=\sum_{j=1}^c d_j(t)$ for each $t$.
By passing to a subsequence, we may assume that $(a_0(t),...,a_n(t))$ converges to $(A_0,...,A_n)$.
We define the subset $Z:=\{i=0,...,n|A_i=0\}$.
From Proposition \ref{comparison} and the condition $\sum_{i=0}^n A_i=1$, we obtain the order $|Z|\le m:=\dim X$.

For each $i,t$, write $a_i(t)=A_i+b_i(t)$. The condition $\sum_{i=0}^n a_i(t)=1$ for all $t$ implies
$$ \sum_{i=0}^n b_i(t)=0 \mbox{ for all }t.$$

We define $I:=\{i=0,1,...,n|b_i(t)<0 \mbox{ for infinite } t\}$.
After passing to a subsequence and renumbering indices of $0,...,n$,
we may assume that $I=\{0,...,\gamma\}$ is a nonempty subset of $\{0,...,n\}$ with the following properties:
\begin{enumerate}
\item $b_i(t)<0$ for all $t$ if and only if $i\in I$.
\item $\frac{A_0}{-b_0(t)}\leq \frac{A_1}{-b_1(t)}\leq \cdots \leq \frac{A_\gamma}{-b_\gamma(t)}$ for all $t$.
\item for all $i\in I$, $b_i(t)$ is strictly increasing as a function of $t$.
\end{enumerate}

We define $\mu (I)$ to be the first nonnegative number $\mu$ so that $\mbox{P}_{D}\cap X(t)\neq \emptyset$ with $D:= \{0,...,\mu-1,\mu\}\subseteq I$ for infinite $t$
but $\mbox{P}_{E}\cap X(t)= \emptyset$ for all subsets $E$ of $\{0,...,\mu-1\}$ for all sufficient large $t$.
In this case, we may assume $\mbox{P}_{D}\cap X(t)\neq \emptyset$ for all $t$ by passing to a subsequence.
If $\mbox{P}_{I}\cap X(t)= \emptyset$ for all large $t$, then we put $\mu (I)=\gamma+1=|I|$. However, this does not occur due to Lemma \ref{mu} below.

\begin{rem}\label{mu<c}
One sees $\mu< c$ from the definition of $I$. Indeed, suppose $\mu\geq c$,
then for all sufficient large $t$ and for $j=1,..., c$, $f_j(t)$ has a monomial $\Pi_{i\in \{0,...,c-1\}}\ x_i^{k_{j,i}(t)}$.
By counting degrees, this gives
$$1=\sum_{j=1}^{c}d_j(t)=\sum_{j=1}^c\sum_{i=0}^{c-1} k_{j,i}(t)a_i(t).$$
Since $A_i>0$ for all $i\in I$ and each $k_{j,i}(t)$ is a nonnegative integer, $k_{j,i}(t)$ is bounded from above for all large $t$ .
After passing to a subsequence, we may assume that $k_{j,i}(t)=k_{j,i}$ is independent of $t$ for each $i,j$.
In particular, $K_i(t):=\sum_{j=1}^c k_{j,i}$ is independent of $t$.
Hence $1=\sum_{j=1}^{c}d_j(t)=\sum_{i=0}^{c-1}K_i(A_i+b_i(t))$.
This gives a contradiction that $1=\sum_{i=0}^{c-1} K_i A_i$ and $0=\sum_{i=0}^{c-1} K_ib_i(t)<0$.
\end{rem}

For every non-empty subset $E$ of $\{0,1,...,n\}$ and for all polynomial $f(x_0,...,x_n)$, denote by $f^E$ to be the polynomial $f(x_0,...,x_n)|_{x_i=0\ \forall i\notin E}$.
We need some positivities.

\begin{lem} \label{positivities}
Let $f_1,...,f_c$ be polynomials in the variables $x_0,...,x_n$.
Assume $\mu>0$ and for all subset $E$ of $\{0,...,\mu-1\}$,
there exist at least $|E|$ polynomials of $\{f_j^E\}$ being non-zero.
Then up to rearranging the indices of $1,...,c$, $f_j^{\{0,...,\mu-1\}}$ has a monomial involving $x_{j-1}$ for all $j=1,...,\mu$.
\end{lem}
\begin{proof}
For each $j=1,...,c$ and subset $E$ of $\{0,1,...,n\}$, we define a subset $R_j^E:=\{\ i\ |\ f_j^E \mbox{ has a monomial involving the variable }x_i \}.$
We need to show that up to renumbering the indices of $1,...,c$,  $j-1\in R_j^{\{0,...,\mu-1\}}$ for all $j=1,...,c$.

We prove by induction on $\mu$. It is obvious for $\mu=1$. We assume that the statement holds for $\mu -1$.
By rearranging the indices of $1,...,c$, we may assume that $j-1 \in R_j^{\{0,...,\mu-2\}}$ for all $j=1,...,\mu-1$.
For all $\lambda=\mu,...,c$, we may assume that $\mu-1\not\in R_\lambda^{\{0,....,\mu-1\}}$.
Otherwise, by replacing $\lambda$ by $\mu$, the statement is true.

So there is an integer $\sigma(1)<\mu$ such that $R_{\sigma(1)}^{\{\mu-1\}}\neq \emptyset$.
If $R_{\lambda}^{\{\sigma(1)-1,\mu-1\}}\neq \emptyset$ for some $\lambda\geq \mu$, we put $\omega=0$.
Otherwise, by the assumption, there exists a $\sigma(2)\in \{1,...,\mu-1\}-\{\sigma(1)\}$ with $R_{\sigma(2)}^{\{\sigma(1)-1,\mu-1\}}\neq \emptyset$.
If $R_{\lambda}^{\{\sigma(1)-1,\sigma(2)-1,\mu-1\}}\neq \emptyset$ for some $\lambda\geq \mu$, we put $\omega=1$.
Otherwise, there exists a $\sigma(3)\in \{1,...,\mu-1\}-\{\sigma(1),\sigma(2)\}$ with $R_{\sigma(3)}^{\{\sigma(1)-1,\sigma(2)-1,\mu-1\}}\neq \emptyset$.
Combining the assumption that there exists at least $\mu$ nonempty subsets $R_j^{\{0,...,\mu-1\}}$ among all subsets $R_1^{\{0,...,\mu-1\}},...,R_c^{\{0,...,\mu-1\}}$,
we can define inductively an integer $\omega<\mu-1$ and distinct positive integers $\sigma(1),...,\sigma(\omega+1)< \mu$ satisfying the following two conditions:
\begin{enumerate}
\item all the subsets $R_{\sigma(1)}^{\{\mu-1\}},...,R_{\sigma(\omega+1)}^{\{\sigma(1)-1,...,\sigma(\omega)-1,\mu-1\}}$ are not empty;
\item $R_\lambda^{\{\sigma(1)-1,...,\sigma(\omega)-1,\mu-1\}}=\emptyset$ for all $\lambda\geq \mu$.
But $R_{\lambda}^{\{\sigma(1)-1,...,\sigma(\omega+1)-1,\mu-1\}}$ contains an element $\sigma(\omega+1)-1$ for some $\lambda\geq \mu$.
\end{enumerate}

\begin{claim}
For all $i\in\{\sigma(1)-1,...,\sigma(\omega+1)-1,\mu-1\}$,
there exists a bijection
$v_{i}:\{1,...,\omega+1\}\ \rightarrow \{\sigma(1)-1,...,\sigma(\omega+1)-1,\mu-1\}-\{i\}$
such that $v_i(j)\in R_{\sigma(j)}^{\{0,...,\mu-1\}}$ for all $j=1,...,\omega+1$.
\end{claim}
\begin{proof}
We prove this by induction on $\omega$. This is clear when $\omega=0$. Suppose the statement is true for $\omega-1$.
If $i\neq\sigma(\omega+1)-1$, by induction hypothesis we have a bijection $v_i$ from $\{j=1,...,\omega\}$ to $\{\sigma(1)-1,...,\sigma(\omega)-1,\mu-1\}-\{i\}$ satisfying that
$v_i(j)\in R_{\sigma(j)}^{\{0,...,\mu-1\}}$ for all $j=1,...,\omega.$
We obtain the result by setting $v_i(\omega+1)=\sigma(\omega+1)-1$ since $\sigma(\omega+1)-1\in R_{\sigma(\omega+1)}^{\{0,...,\mu-1\}}$.

If $i=\sigma(\omega+1)-1$, condition (1) shows that $R_{\sigma(\omega+1)}^{\{\sigma(1)-1,...,\sigma(\omega)-1,\mu-1\}}$ contains an element, say $\tau$.
Note that $\tau\neq \sigma(\omega+1)-1$.
By induction there exists a bijection $v_\tau$ from $\{j=1,...,\omega\}$ to $\{\sigma(1)-1,...,\sigma(\omega)-1,\mu-1\}-\{\tau\}$ satisfying
$v_\tau(j)\in R_{\sigma(j)}^{\{0,...,\mu-1\}}$ for all $j=1,...,\omega.$
Put $v_i(j)=v_\tau(j)$ for $j=1,...,\omega$ and $v_i(\omega+1)=\tau$. We prove this claim.
\end{proof}
From condition (2) and the claim, $\sigma(\omega+1)-1\in R_\lambda^{\{0,...,\mu-1\}}$ for some $\lambda>\mu-1$ and there exists a bijection
$v_{\sigma(\omega+1)-1}:\{1,...,\omega+1\}\ \rightarrow \{\sigma(1)-1,...,\sigma(\omega)-1,\mu-1\}$
such that $v_{\sigma(\omega+1)-1}(j)\in R_{\sigma(j)}^{\{0,...,\mu-1\} }$ for all $j=1,...,\omega+1$.
Since $R_j^{\{0,...,\mu-1\}}$ is assumed to contain the element $j-1$ for all $j\in\{\omega+2,...,\mu-1\}-\{\sigma(1),...,\sigma(\omega+1)\}$,
by replacing $\lambda$ by $\mu$ and then renumbering the indices of $\sigma(1),...,\sigma(\omega+1),\mu$, we obtain the lemma.
\end{proof}

Suppose $\mu>0$.
Since $\mbox{P}_{\{0,...,\mu-1\}}\cap X(t)=\emptyset$ for all large $t$, condition (1) of Proposition \ref{quasismooth} holds.
By passing to a subsequence and renumbering indices of $1,...,c$,
we may assume that all polynomials $f_1^{\{0,...,\mu-1\}}(t),...,f_\mu^{\{0,...,\mu-1\}}(t)$ are not identically zero.
As the discussion in Remark \ref{mu<c} and by choosing a subsequence, there exist nonnegative integers $k_{j,i}$ independent of $t$ such that for all $t$,
\[
\begin{array}{llll}
f_1(t) &= \Pi_{i\in \{0,...,\mu-1\}}\ x_i^{k_{1,i}}+\mbox{others},\\
     &\vdots           \\
f_{\mu}(t)&= \Pi_{i\in \{0,...,\mu-1\}}\ x_i^{k_{\mu,i}}+\mbox{others}.
\end{array}
\]
Furthermore, we may assume $k_{j,j-1}\ge 1$ for all $j=1,...,\mu$ from following induction process in the proof of Lemma \ref{positivities}.
We shall show $\mbox{P}_{I}\cap X(t)\neq \emptyset$ for infinitely many $t$ by the following:

\begin{lem}\label{mu}
$\mu<\gamma+1=|I|.$
\end{lem}
\begin{proof}
Suppose on the contrary that $\mu=\gamma+1$. For infinite $t$, we may assume that for $j=1,...,\gamma+1$,
\[ f_{j}(t)= \Pi_{i\in \{0,...,\gamma\}}\ x_i^{k_{j,i}}+\mbox{others},
\]
where the nonnegative integers $k_{j,i}$ are independent of $t$ and $k_{j,j-1}\ge 1$.
Since $b_i(t)$ is strictly increasing for all $i\in I$ as a function of $t$, by taking a subsequence and counting degrees,
we may assume that each $d_1(t),...,d_{\gamma+1}(t)$ is strictly increasing.
On the other hand, we define $F:=\{0,...,n\}-(I\cup Z)$.
From the condition $\sum_{i=0}^n A_i=1$ and the upper bound $|Z|\leq m=\dim X$, $F$ is not empty from Remark $\ref{mu<c}$.
Since each $d_1(t),...,d_{\gamma+1}(t)$ is strictly increasing, the polynomials $f^F_1(t),...,f^F_{\gamma+1}(t)$ are all identically zero for all large $t$, i.e,
\[
\mbox{P}_F\cap X(t)=\mbox{P}_F\cap (f_{\gamma+2}(t)=\cdots =f_{c}(t)=0).
\]

Hence $\mbox{P}_F\cap X(t)$ has dimension $\geq n-(\gamma+1)-m-(c-(\gamma+1))=0$.
From the quasismoothness on $\mbox{P}_F\cap X(t)$, condition (2) of Proposition \ref{quasismooth} holds for all large $t$.
In particular, for $j=1,...,\gamma+1$, $f_j(t)$ has monomial $x_{e_j}\Pi_{i\in F}\ x_i^{q_{j,i}}$ for infinite $t$ where
the nonnegative integers $q_{j,i},e_j$ are independent of $t$ and
$e_1,...,e_{\gamma+1}$ are mutually distinct.
By counting degrees, this shows that for $j=1,...,\gamma+1$,
\[d_j(t)=a_{e_j}+\sum_{i\in F} q_{j,i}a_i(t)\mbox{ for some }e_j\in I\cup Z.
\]
If $e_j\notin I$ for some $j=1,...,\gamma+1$, we see that $d_j(t)$ is not strictly increasing as a function of $t$.
In this case, we get the desired contradiction.
The remaining case is that $e_j\in I$ for all $j=1,...,\gamma+1$.
From the same reason, the function
$$\sum_{j=1}^{\gamma+1}(d_j(t)-a_{j-1}(t))=\sum_{j=1}^{\gamma+1}d_j(t)-\sum_{j=1}^{\gamma+1}a_{e_j}(t)=\sum_{j=1}^{\gamma+1}\sum_{i\in F}q_{j,i}a_i(t)$$
is not increasing.
On the other hand,
$\sum_{j=1}^{\gamma+1}(d_j(t)-a_{j-1}(t))$ is also strictly increasing by the assumption that $X(t)$ is not an intersection of a linear cone and another subvariety
and the positivity $k_{j,j-1}\geq 1$ for all $j=1,...,\gamma+1$.
We derive a contradiction and prove the lemma.
\end{proof}

From the definition of $\mu$ and Lemma \ref{mu}, the quasismoothness on $\mbox{P}_{D}\cap X(t)$ implies condition (2) of Proposition \ref{quasismooth}
where $D$ is the stratum $\{0,...,\mu-1,\mu\}$.
In particular, by passing to a subsequence and rearranging indices of $\mu +1,...,c$,
there exist integers $s$ with $\mu\leq s<c$ and $k_{j,i},e_j$ independent of $t$ such that
\[ \left\{
\begin{array}{ll}
f_j(t)=\Pi_{i\in D}\ x_i^{k_{j,i}}+\mbox{others} & \mbox{for } \mu < j \leq s, \\
f_j(t)=x_{e_j}\Pi_{i\in D}\ x_i^{k_{j,i}}+\mbox{others},\mbox{ with } e_j\ge \mu+1 & \mbox{for } s< j \leq c.
\end{array} \right. \]
Here the integers $e_{s+1},...,e_c$ are mutually distinct and also $s=0$ if and only if $\mu=0$ from the definition.
Note that such an integer $s$ with $s<c$ exists because some of the degree functions $d_1(t),...,d_c(t)$ is not strictly increasing.
In particular, by counting degrees of $f_1(t),...,f_{c}(t)$ we obtain
\[1=\sum_{j=1}^{c}\ d_j(t)=\sum_{i=0}^\mu K_i a_i(t)+\sum_{i=s+1}^{c}\ a_{e_i}(t)\ \ \mbox{for infinite }t,  \eqno{(\dag\ 1)}\]
where $K_i:=\sum_{j=1}^c k_{j,i}\geq 1$ for $i=0,...,\mu-1$.
One has the inequality:
\begin{align*}
& \sum_{i=0}^{\mu-1}\ (1-K_i)b_i(t)+\sum_{i\in I-\{0,...,\mu-1,e_{s+1},...,e_c\}} \ b_i(t) \\
&=-\sum_{i=0}^{\mu-1}\ K_ib_i(t)-\sum_{i\in I\cap \{e_{s+1},...,e_c\}} b_i(t)-\sum_{i\notin I}b_i(t) \\
&\leq -\sum_{i=0}^{\mu-1}\ K_ib_i(t)-\sum_{i\in \{e_{s+1},...,e_c\}} b_i(t)= K_{\mu} b_{\mu}(t).
\hspace{3.6cm}    \dotfill (\dag\ 2)
\end{align*}
So
\begin{align*}
K_\mu A_\mu &\leq \sum_{i=0}^{\mu-1}\ (1-K_i)b_i(t)\frac{A_\mu}{b_{\mu}(t)}+\sum_{i\in I-\{0,...,\mu-1,e_{s+1},...,e_c\}} \ b_i(t)\frac{A_\mu}{b_{\mu}(t)} \notag \\
&\leq \sum_{i=0}^{\mu-1}\ (1-K_i)A_i+\sum_{i\in I-\{0,...,\mu-1,e_{s+1},...,e_c\}} \ A_i.   \hspace{2.8cm}  \dotfill (\dag\ 3) \\
\end{align*}
From $(\dag\ 1)$ and $(\dag\ 3)$, we have

\begin{align*}
1&=\sum_{i=0}^{\mu-1}\ K_iA_i+ K_{\mu}A_{\mu}+\sum_{i=s+1}^{c} A_{e_i}  \\
&\leq \sum_{i=0}^{\mu-1}\ (K_i+(1-K_i))A_i+\sum_{i\in I-\{0,...,\mu-1,e_{s+1},...,e_c\}} \ A_i +\sum_{i=s+1}^{c} A_{e_i}  \\
&\leq \sum_{i=0}^{\mu-1}\ A_i+\sum_{i=\mu}^{n}\ A_i=1.  \hspace{7cm}    \dotfill (\dag\ 4)\\
\end{align*}

Hence all inequalities in $(\dag\ 2),(\dag\ 3),(\dag\ 4)$ are actually equalities.
For all $k\notin I\cup \{e_{s+1},...,e_c\}$, from $(\dag\ 4)$ and $(\dag\ 2)$,
we have $A_k=0$ and $b_k(t)=0$ for infinite $t$ respectively.
This gives a contradiction to $a_k(t)=A_k+b_k(t)=0$ for infinite $t$.
So $I\cup \{e_{s+1},...,e_c\}=\{0,...,n\}$.
From $(\dag\ 3)$, we get the equalities: for infinite $t$,
\[
\frac{-b_0(t)}{A_0}=\cdots=\frac{-b_{\mu}(t)}{A_{\mu}}=\frac{-b_k(t)}{A_k},\ k\in I-\{e_{s+1},...,e_c\}.
\]
Hence we obtain a subsequence
\[
(a_0(t),...,a_n(t))=(A_0(1-b(t)),...,A_{m+s}(1-b(t)),a_{m+s+1}(t),...,a_n(t)),
\]
where $b(t):=\frac{-b_{\mu}(t)}{A_{\mu}}>0$.
By rearranging the indices of $m+s+1,...,n$, we may assume $e_j=m+j$ for all $j=s+1,...,c$.
Again, after passing to a subsequence and rearranging indices of $m+s+1,...,c$,
we define $p\geq s$ to be the integer satisfying for all $t$,
\[  \left\{
\begin{array}{ll}
a_{i}(t)= A_i(1-b(t))\mbox{ for all }i\leq m+p, \\
a_{i}(t)\neq A_i(1-b(t))\mbox{ for all }i> m+p.
\end{array} \right. \]
From the choice of $p$, we observe $\{0,...,m+p\}\subseteq I$.
In particular, we obtain $p<c$ since by definition $\sum_{i=0}^n b_i(t)=0$.
\begin{claim}
For all $j> p$ and sufficient large $t$, every monomial of $f_j(t)$ involves at least one of the variables $x_{m+p+1},...,x_n$.
\end{claim}
\begin{proof}
Since $\{0,...,m+p\}$ is a subset of $I$, $A_i>0$ for all $i=0,...,m+p$.
Suppose there is a monomial $\Pi_{i\in \{0,...,m+p\}}\ x_i^{q_{j,i}}$ of $f_j(t)$ for infinite $t$ for some fixed $j>p$.
From the construction of $e_j$, $f_j(t)=x_{e_j}\Pi_{i\in D}\ x_i^{k_{j,i}}+\mbox{others}$.
By counting degrees, we see that
\begin{align*}
a_{m+j}(t)&=a_{e_j}(t)=d_j(t)-\sum_{i\in D}k_{j,i}A_i(1-b(t))\\
&=(1-b(t))(\sum_{i=0}^{m+p} q_{j,i}A_i-\sum_{i\in D}k_{j,i}A_i).
\end{align*}
This contradicts with the definition of $p$.
\end{proof}

For sufficient large $t$ and for $j=p+1,...,c$, we write
\[
f_j(t)=\sum_{i=m+p+1}^n g_j^i(t)(x_0,...,x_{m+p})x_i+l_j(t)(x_0,...,x_n),
\]
where 
$\deg_{x_{m+p+1},...,x_n}(l_j(t)(x_0,...,x_n))\geq 2.$
Define $E_p:=\{0,...,m+p\}$ and $G(t)$ to be the matrix
\[
\left(
  \begin{array}{ccc}
    g_{p+1}^{m+p+1}(t) & \cdots & g_{p+1}^n(t) \\
    \vdots &  &\vdots  \\
    g_{c}^{m+p+1}(t) & \cdots & g_{c}^n(t) \\
  \end{array}
\right).
\]
On the general points in the affine cone of $X(t)\cap \mbox{P}_{E_p}$, the Jacobian matrice are of the form
\[
\left(
  \begin{array}{cc}
    * &  *  \\
    O & G(t) \\
  \end{array}
\right).
\]
Then for sufficient large $t$,
the non-quasismooth locus of $X(t)$ contains $\mbox{P}_{E_p} \cap X(t)\cap (\det G(t)=0)=\mbox{P}_{E_p}\cap (f_1(t)=\cdots =f_{p}(t)=0)\cap (\det G(t)=0)$ which has dimension $\geq m+p-p-1 \geq 0$.
We get the contradiction and therefore prove Theorem \ref{CY}.
\end{proof}

\end{document}